\documentclass[a4paper,12pt]{article}
\usepackage{amsmath, amsfonts, amssymb, amsthm}
\usepackage[english]{babel}

\newcounter{num}[section]

\newenvironment{theorem}
{\refstepcounter{num}%
\bigskip\noindent\nopagebreak[4]{\bf Theorem~\arabic{section}.\arabic{num}. }\it}
\newenvironment{corollary}
{\refstepcounter{num}%
\bigskip\noindent\nopagebreak[4]{\bf Corollary~\arabic{section}.\arabic{num}. }\it}
\newenvironment{lemma}
{\refstepcounter{num}%
\bigskip\noindent\nopagebreak[4]{\bf Lemma~\arabic{section}.\arabic{num}. }\it}
\newcommand{\LL}{{\mathcal{L}}}

\newcommand{\Ss}{{\mathbf{S}}}
\newcommand{\V}{{\mathrm{V}}}

\newcommand{\pr}{{\prime}}

\newcommand{\M}{{\mathcal{M}}}

\renewcommand{\P}{{\mathbf{P}}}
\renewcommand{\c}{{\mathbf{c}}}

\newcommand{\one}{{\mathbf{1}}}

\newcommand{\1}{{^{-1}}}

\begin{document}

\author{Artem N. Shevlyakov \footnote{The author was supported by RFII grant 14-01-00068}}
\title{On disjunctions of algebraic sets in completely simple semigroups}

\maketitle

\abstract{A semigroup $S$ is called an equational domain if any finite union of algebraic sets over $S$ is algebraic. We give some necessary and sufficient conditions for a completely simple semigroup to be an equational domain.}

\section{Introduction}

G. Baumslag, A. Miasnikov and V.~Remeslennikov~\cite{AG_over_groupsI,AG_over_groupsII} laid some basics of algebraic geometry over groups. Recall that an equation over a group $G$ is an equality $w(X)=1$, where $w(X)$ is an element of the free product $G\ast F(X)$. An algebraic set over $G$ is the set of all solutions of a system of equations over $G$. 

Finite unions of algebraic sets are not necessarily algebraic over $G$. However, there exist groups (which were completely described in the paper written by E.~Daniyarova, A.~Miasnikov and V.~Remeslennikov~\cite{uniTh_IV}) where any finite union of algebraic sets is algebraic. Following~\cite{uniTh_IV}, the groups with such property are called equational domains.  

One can similarly pose the problem of existence of equational domains in semigroups. To solve this problem it is natural to search equational domains in the varieties of semigroups which are close to groups. One such ``group-like'' class is that of completely simple semigroups, since any completely simple semigroup is a disjoint union of copies of a group. 

Any completely simple semigroup admits operation of the inversion ${}^{-1}$, hence we shall consider completely simple semigroups as algebraic structures in the language $\LL=\{\cdot,{}^{-1}\}$. The use of the language $\LL$ instead of $\{\cdot\}$ is quite convenient, since the class of completely simple semigroups forms a variety in the language $\LL$.

In the current paper we prove the necessary and sufficient conditions for a completely simple semigroup $S$ to be an equational domain (Theorem~\ref{th:criterion}). From the obtained criterion it follows that  free completely simple semigroups of rank $n\geq 2$ and any free product of arbitrary completely simple semigroups are equational domains (similar results hold in the variety of groups). Moreover, in Section~\ref{sec:generalization} we extend Theorem~\ref{th:criterion} to other semigroup languages.

\section{Notions of semigroup theory}

A subset $I\subseteq S$ is called a \textit{left (right) ideal} if $sa\in I$ (respectively, $as\in I$) for any $s\in S$, $a\in I$. If $I$ is a right and left ideal simultaneously, we say that $I$ is a {\it two-sided ideal} (or {\it an ideal} for shortness).  

A semigroup $S$ with a unique ideal $I=S$ is called \textit{simple}. A simple semigroup with minimal left and right ideal is \textit{completely simple (c.s.)}.

The next classic theorem completely describes c.s. semigroups.

\begin{theorem}
\label{th:rees}
For any c.s. semigroup $S$ there exists a group $G$ and sets $I,\Lambda$ such that  $S$ is isomorphic to the set of triples $(\lambda,g,i)$, $g\in G$, $\lambda\in\Lambda$, $i\in I$ with multiplication
\[
(\lambda,g,i)(\mu,h,j)=(\lambda,gp_{i\mu}h,j),
\]
where $p_{i\mu}\in G$ is an element of a matrix $\P$ such that
\begin{enumerate}
\item $\P$ consists of  $|I|$ rows and $|\Lambda|$ columns;
\item the matrix $\P$ is \textit{normalised}, i.e. 
\[p_{1i}=p_{\lambda 1}=1\in G\mbox{ for all }\lambda\in\Lambda,\;i\in I.\]
\end{enumerate}
\end{theorem}

Following Theorem~\ref{th:rees}, we denote any c.s. semigroup $S$ by $S=(G,\P,\Lambda,I)$. The group $G$ and the matrix $\P$ are called \textit{the structural group} and~\textit{sandwich-matrix}, respectively. The elements $\lambda,i$ occurring in a triple $(\lambda,g,i)\in S$ are \textit{the first and the second indexes}, respectively. 

By Theorem~\ref{th:rees}, any c.s. semigroup $S$ is a disjoint union of copies of the structural group  $G$. Clearly, the identity elements of the maximal subgroups are $(\lambda,p_{i\lambda}\1,i)$, $\lambda\in\Lambda$, $i\in I$. The inversion ${}\1$ in the subgroup defined by the indexes $\lambda,i$ is given by 
\[
(\lambda,g,i)\1=(\lambda,p_{i\lambda}\1g\1p_{i\lambda}\1,i).
\]

The class of c.s. semigroups is a variety in the language $\{\cdot,{ }\1\}$, since it is defined by the identities:
\[
xx\1 x=x,\; xx\1=x\1 x,\; (x\1)\1=x,\; (xyx)\1(xyx)=x\1 x.
\]

The structure of free c.s. semigroups is described by the next theorem.

\begin{theorem}\textup{\cite{clifford_free,rasin}}
\label{th:free_CSS_structure}
Let $X=\{x_1,x_2,\ldots,x_n\}$, $Y=\{y_{i\lambda}|i\in I,\lambda\in\Lambda\}$ be finite sets of letters, and $I=\Lambda=\{1,2,\ldots, n\}$. Let $F(X\cup Y)$ denote the free group generated by the set $X\cup Y$. Then, the free c.s. semigroup $F_{css}(X)$ generated by $X$ is defined by  $F_{css}=(F(X\cup Y),\P,I,\Lambda)$, where $\P=(y_{i\lambda})$ and the generators $x_i$ correspond to the triples $(i,x_i,i)\in F_{css}(X)$.  
\end{theorem}

It is easy to check that the free c.s. semigroup of rank $1$ is isomorphic to the infinite cyclic semigroup of the same rank.

One can define the free product in the variety of c.s. semigroups. The representation of the free product of two c.s. semigroups is given in the next theorem.

\begin{theorem}\textup{\cite{jones}}
\label{th:free_product_CSS_structure}
Let $S=S_1\ast S_2$ be the free product of two c.s. semigroups $S_1=(G_1,\P_1,\Lambda_1,I_1)$, $S_2=(G_2,\P_2,\Lambda_2,I_2)$. Then, $S=(G,\P,\Lambda,I)$, where 
\begin{enumerate}
\item $\Lambda=\Lambda_1\sqcup\Lambda_2$, $I=I_1\sqcup I_2$;
\item the elements of the sandwich-matrix $\P=(p_{i\lambda})$ are
\[
p_{i\lambda}=
\begin{cases}
y_{i\lambda},\mbox{ если }i\in I_1, \lambda\in\Lambda_2\mbox{ or }i\in I_2, \lambda\in\Lambda_1\\
p_{i\lambda}^{(1)},\mbox{ если }i\in I_1, \lambda\in\Lambda_1\\
p_{i\lambda}^{(2)},\mbox{ если }i\in I_2, \lambda\in\Lambda_2\\
\end{cases}
\]  
where $p_{i\lambda}^{(j)}$ is the element of $\P_j$ with the indexes $i,\lambda$.
\item
\[
G=G_1\ast G_2\ast F(Y),
\]
where $F(Y)$ is the free group generated by the set
\[
Y=\{y_{i\lambda}|i\in I_1, \lambda\in\Lambda_2\mbox{ or }i\in I_2, \lambda\in\Lambda_1\}.
\]
\end{enumerate}

\end{theorem}

\section{Algebraic geometry over semigroups}

All definitions below are derived from the general notions of~\cite{uniTh_I}, where the notions of algebraic geometry were formulated for an arbitrary algebraic structure in the language with no predicates.

Since the inversion ${ }\1$ is an algebraic operation in any c.s. semigroup, we consider the language $\LL_0=\{\cdot,\1\}$. For a given c.s. semigroup $S$ one can extend $\LL_0$ by new constants $\{s|s\in S\}$ which correspond to the elements of $S$. The obtained language is denoted by $\LL_S$ and in the sequel all semigroups are considered in this language.

Let $X$ be a finite set of variables $x_1,x_2,\ldots,x_n$. \textit{A term} of a language $\LL_S$ ($\LL_S$-term) in variables $X$ is one of the following expressions:
\begin{enumerate}
\item variable $x_i$;
\item constant $s$;
\item a product of two terms;
\item $(t(X))\1$, where $t(X)$ is a term.
\end{enumerate}

For example, the expressions $xs(y^2x)\1$, $(xs_1\1 y)\1 s_2x^2$ are $\LL_S$-terms.

{\it An equation} over $\LL_S$ is an equality of two $\LL_S$-terms $t(X)=s(X)$. {\it A system of equations} over $\LL_S$ ({\it a system} for shortness) is an arbitrary set of equations over $\LL_S$.
 
A point $P=(p_1,p_2,\ldots,p_n)\in S^n$ is a \textit{solution} of a system $\Ss$ in variables $x_1,x_2,\ldots,x_n$, if the substitution $x_i=p_i$ reduces each equation of $\Ss$ to a true equality in the semigroup $S$. The set of all solutions of a system $\Ss$ in the semigroup $S$ is denoted by $\V_S(\Ss)$. A set $Y\subseteq S^n$ is called  {\it algebraic} over the language $\LL_S$ if there exists a system over $\LL_S$ in variables $x_1,x_2,\ldots,x_n$ with the solution set $Y$. 

Two systems are called \textit{equivalent} over a c.s. semigroup $S$ if they have the same solution in $S$.

Following~\cite{uniTh_IV}, let us give the main definition of our paper.

A c.s. semigroup $S$ is an {\it equational domain} ({\it e.d.} for shortness) in the language $\LL_S$ if for any finite set of algebraic sets $Y_1,Y_2,\ldots,Y_n$ the union $Y=Y_1\cup Y_2\cup\ldots\cup Y_n$ is algebraic. 

The next theorem contains necessary and sufficient conditions for a semigroup to be an e.d.  

\begin{theorem}\textup{\cite{uniTh_IV}}
\label{th:about_M}
A semigroup $S$ in the language $\LL_S$ is an e.d. iff the set 
\[
%\label{eq:set_M}
\M_{sem}=\{(x_1,x_2,x_3,x_4)|x_1=x_2\mbox{ or }x_3=x_4\}\subseteq S^4
\]
is algebraic, i.e. there exists a system $\Ss$ in the variables $x_1,x_2,x_3,x_4$ with the solution set $\M_{sem}$.
\end{theorem}

Below we study equations over groups, therefore we should give some definitions of algebraic geometry over groups. Any group $G$ below will be considered in the language $\LL_G=\{\cdot,^{-1},1\}\cup\{g|g\in G\}$ where constants of $\LL_G$ correspond to elements of $G$. \textit{An $\LL_G$-term}  in the variables  $X=\{x_1,x_2,\ldots,x_n\}$ is defined in the same way as it is the definition of a term over c.s. semigroup. However, in groups the identity $(xy)\1=x\1 y\1$ takes place, hence any $\LL_G$-term over a group $G$ is equivalent to a finite product that consists of the variables in integer powers and constants $g\in G$. In other words, an $\LL_G$-term is an element of the free product $F(X)\ast G$, where $F(X)$ is a free group generated by the set $X$. 

The definitions of equations, algebraic sets and equational domains over groups are similar to the corresponding definitions in the semigroup case. 

For the groups of the language $\LL_G$ we have the following result.

\begin{theorem}\textup{\cite{uniTh_IV}}
A group $G$ in the language $\LL_G$ is an e.d. iff the set 
\label{th:criterion_for_groups}
\begin{equation*}
\M_{gr}=\{(x_1,x_2)|x_1=1\mbox{ or }x_2=1\}\subseteq G^2
\end{equation*}
is algebraic, i.e. there exists a system $\Ss$ in variables $x_1,x_2$ with the solution set $\M_{gr}$.
\end{theorem}

One can reformulate Theorem~\ref{th:criterion_for_groups} in a simpler  form using the following definition. An element $x\neq 1$ of a group $G$ is a \textit{zero-divisor} if there exists $1\neq y\in G$ such that for any $g\in G$ it holds $[x,y^g]=1$ (here $y^g=gyg^{-1}$, $[a,b]=a^{-1}b^{-1}ab$).

\begin{theorem}\textup{\cite{uniTh_IV}}
\label{th:zero_divisors}
A group $G$ in the language $\LL_G$ is an e.d. iff it does not contain zero-divisors.
\end{theorem}

Using Theorem~\ref{th:zero_divisors}, one can obtain the next properties of free objects in the variety of groups.
 
\begin{corollary}\textup{\cite{uniTh_IV}}
\label{cor:free_group_is_ED}
Every non-abelian free group $G$ is an e.d. in the group language $\LL_G$ (this was initially proved by G.~Gurevich, see the proof in~\cite{makanin}). 
\end{corollary}

\begin{corollary}\textup{\cite{uniTh_IV}}
\label{cor:free_product_of_groups_is_ED}
Every free product $G=G_1\ast G_2$, except for $G=\mathbb{Z}_2\ast\mathbb{Z}_2$ ($\mathbb{Z}_2$ is the cyclic group of the order $2$), is an e.d. in the language $\LL_G$. 
\end{corollary}

\section{Main results}

A sandwich-matrix $\P$ is said to be \textit{non-singular} if it does not contain two equal rows or columns.

Let $s_1,s_2$ be two distinct elements of a c.s. semigroup $S$. We say that an $\LL_S$-term $t(x)$ \textit{separates} $s_1,s_2$ if $t(s_1)\neq t(s_2)$.

\begin{lemma}\textup{(Lemma 4.2 of \cite{ED_I})}
\label{l:exists_dist_term}
If a c.s. semigroup $S=(G,\P,\Lambda,I)$ has a non-singular sandwich-matrix $\P$, then for any pair of distinct elements $s_1,s_2\in S$ there exists an $\LL_S$-term
\begin{equation}
t(x)=(1,1,i)x(\lambda,1,1),
\label{eq:dist_term}
\end{equation}
separating $s_1,s_2$ for some $\lambda\in\Lambda$, $i\in I$.
\end{lemma}

Let $t(x)$ be an $\LL_S$-term. By $[t](x)$ we denote the $\LL_S$-term which is obtained from $t(x)$ by omitting all occurrences of the inversion. For example, if $t(x)=((xsy\1)\1zx\1)\1$ then $[t](x)=xsyzx$. 

\begin{lemma}
\label{l:exists_2_non_dist_elems}
Suppose the sandwich-matrix $\P$ of a c.s. semigroup $S=(G,\P,\Lambda,I)$ has equal rows (columns) with numbers $i,j$ (respectively, $\lambda,\mu$). Then for the elements $s_1=(1,1,i)$, $s_2=(1,1,j)$ (resp. $s_1=(\lambda,1,1)$, $s_2=(\mu,1,1)$) and an arbitrary $\LL_S$-term $t(x)$ one of the following holds:
\begin{enumerate}
\item $t(s_1)=t(s_2)$;
\item $t(s_1)=(\nu,g,i)$, $t(s_2)=(\nu,g,j)$ for some $g\in G$, $\nu\in\Lambda$, if $[t](x)$ ends with the variable $x$ (resp. $t(s_1)=(\lambda,g,k)$, $t(s_2)=(\mu,g,k)$ for some $g\in G$, $k\in I$, if  $[t](x)$ begins with the variable $x$). In other words, the elements $t(s_1),t(s_2)$ are distinguished only by the second (respectively, first) index.
\end{enumerate}
\end{lemma}
\begin{proof}
Assume that the $\lambda$-th and $\mu$-th columns of $\P$ are equal.(similarly, one can consider $\P$ with equal rows).

We prove this lemma by induction on the construction of a term $t(x)$. If $t(x)$ is a either constant or a variable, lemma obviously holds.

Suppose now $t(x)=(t^\pr(x))\1$, and the lemma holds for $t^\pr(x)$. If $t^\pr(s_1)=t^\pr(s_2)$, then $t(s_1)=t(s_2)$. If $[t(x)]$ begins with the variable $x$, we have
\[
t(s_1)=(\lambda,g,k)\1=(\lambda,p_{k\lambda}\1g\1p_{k\lambda}\1,k),
\]
\[
t(s_2)=(\mu,g,k)\1=(\mu,p_{k\mu}\1g\1p_{k\mu}\1,k).
\]
By the singularity of the matrix $\P$, we have $p_{k\lambda}=p_{k\mu}$, therefore the elements $t(s_1),t(s_2)$ are distinguished only by the first index.

Consider now an $\LL_S$-term $t(x)=t^\pr(x)t^{\pr\pr}(x)$. We have exactly four cases.

\begin{enumerate}
\item Neither of the terms $[t^{\pr}](x)$, $[t^{\pr\pr}](x)$ begins with the variable $x$. By the induction, we have $t^\pr(s_1)=t^\pr(s_2)$, $t^{\pr\pr}(s_1)=t^{\pr\pr}(s_2)$, and, finally, $t(s_1)=t(s_2)$.

\item Only $[t^{\pr\pr}](x)$ begins with the variable $x$. Here we have $t^\pr(s_1)=t^\pr(s_2)=(\nu,h,l)$, $t^{\pr\pr}(s_1)=(\lambda,g,k)$, $t^{\pr\pr}(s_2)=(\mu,g,k)$ for some $g,h\in G, \nu\in\Lambda$, $k,l\in I$.

Thus the values of $t(x)$ equal:
\[
t(s_1)=(\nu,h,l)(\lambda,g,k)=(\nu,hp_{l\lambda}g,k), 
\]
\[
t(s_2)=(\nu,h,l)(\mu,g,k)=(\nu,hp_{l\mu}g,k).
\]
Since $p_{l\lambda}=p_{l\mu}$, we have $t(s_1)=t(s_2)$.

\item Only $[t^{\pr}](x)$ begins with the variable $x$. Therefore, $t^{\pr}(s_1)=(\lambda,g,k)$, $t^{\pr}(s_2)=(\mu,g,k)$, $t^{\pr\pr}(s_1)=t^{\pr\pr}(s_2)=(\nu,h,l)$ for some $g,h\in G, \nu\in\Lambda$, $k,l\in I$. We have 
\[
t(s_1)=(\lambda,g,k)(\nu,h,l)=(\lambda,gp_{k\nu}h,l), 
\]
\[
t(s_2)=(\mu,g,k)(\nu,h,l)=(\mu,gp_{k\nu}h,l),
\]
and the elements $t(s_1),t(s_2)$ are distinguished only by the first index.

\item Both terms $[t^\pr](x)$, $[t^{\pr\pr}](x)$ begin with the variable $x$. By the induction, $t^\pr(s_1)=(\lambda,g,k)$, $t^\pr(s_2)=(\mu,g,k)$, $t^{\pr\pr}(s_1)=(\lambda,h,l)$, $t^{\pr\pr}(s_2)=(\mu,h,l)$ for some $g,h\in G$, $k,l\in I$. Therefore, 
\[
t(s_1)=(\lambda,g,k)(\lambda,h,l)=(\lambda,gp_{k\lambda}h,l),
\]
\[
t(s_2)=(\mu,g,k)(\mu,h,l)=(\mu,g_{k\mu}h,l).
\]
Since $p_{k\lambda}=p_{k\mu}$, the elements $t(s_1),t(s_2)$ are distinguished only by the first index.
\end{enumerate}
\end{proof}

\begin{lemma}
\label{l:singular->not_ED}
If the sandwich-matrix $\P$ is singular, a c.s. semigroup $S=(G,\P,\Lambda,I)$ is not an e.d. in the language $\LL_S$.
\end{lemma}
\begin{proof}
Assume the semigroup $S$ is an e.d. with a singular sandwich-matrix $\P$, and columns with numbers $\lambda,\mu$ of $\P$ are equal (similarly, one can consider a matrix $\P$ with equal rows). 

Consider a system of equations $\Ss(x,y)$ with the solution set $\M=\{(x,y)|x=(\lambda,1,1)\mbox{ or }y=(\lambda,1,1)\}$. Let $t(x,y)=s(x,y)\in \Ss$ be an equation that is not satisfied by the point $((\mu,1,1),(\mu,1,1))\notin\M$. 

Suppose neither of the terms $[t](x,(\mu,1,1))$, $[s](x,(\mu,1,1))$ begins with the variable $x$. Therefore, by Lemma~\ref{l:exists_2_non_dist_elems}, we have 
\[
t((\lambda,1,1),(\mu,1,1))=t((\mu,1,1),(\mu,1,1)),\; s((\lambda,1,1),(\mu,1,1))=s((\mu,1,1),(\mu,1,1)).
\]
Since 
\[
t((\lambda,1,1),(\mu,1,1))=s((\lambda,1,1),(\mu,1,1)),
\]
we obtain 
\[
t(\mu,1,1),(\mu,1,1))=s((\mu,1,1),(\mu,1,1)),
\]
which contradicts the choice of the equation $t(x,y)=s(x,y)$.

Assume now both terms $[t](x,(\mu,1,1))$, $[s](x,(\mu,1,1))$ begin with $x$. By Lemma~\ref{l:exists_2_non_dist_elems}, one can obtain 
\[
t((\lambda,1,1),(\mu,1,1))=(\lambda,g,k),\;t((\mu,1,1),(\mu,1,1))=(\mu,g,k).
\]
Since $((\lambda,1,1),(\mu,1,1))\in \M$, then $s((\lambda,1,1),(\mu,1,1)=(\lambda,g,k)$, and, according Lemma~\ref{l:exists_2_non_dist_elems}, we have 
\[
s((\mu,1,1),(\mu,1,1))=(\mu,g,k),
\]
It follows that
\[
t((\mu,1,1),(\mu,1,1))=s((\mu,1,1),(\mu,1,1)),
\]
which contradicts the choice of $t(x,y)=s(x,y)$.

Thus, we have the last case: the term $[t](x,(\mu,1,1))$ begins with $x$, and $[s](x,(\mu,1,1))$ begins with a constant $\c$. Therefore one of the following holds:
\begin{enumerate}
\item the constant $\c$ was obtained by the substitution of $(\mu,1,1)$ for the variable $y$, i.e. the term $[s](x,y)$ begins with $y$;
\item $[s](x,y)$ begins with $\c$.
\end{enumerate} 
Let us show that both options above are impossible.
\begin{enumerate}
\item Indeed, the equation $t(x,y)=s(x,y)$ is not satisfied by the point $((\lambda,1,1),(\mu,1,1))\in\M$, since the element $t((\lambda,1,1),(\mu,1,1))$ has the first index $\lambda$, whereas  the first index of $s((\lambda,1,1),(\mu,1,1))$ is $\mu$.
\item Suppose $\c=(\nu,g,k)$, so either $t((\lambda,1,1),(\lambda,1,1))\neq s((\lambda,1,1),(\lambda,1,1))$ (if $\nu\neq\lambda$) or $t((\mu,1,1),(\lambda,1,1))\neq s((\mu,1,1),(\lambda,1,1))$ (if $\nu\neq\mu$), since the values of the terms $t(x,y), s(x,y)$ have different first indexes. 
\end{enumerate}
\end{proof}

It is directly checked that the subset
\[
\Gamma=\{(1,g,1)|g\in G\}
\]
of a c.s. semigroup $S=(G,\P,\Lambda,I)$ is a group isomorphic to $G$. Since the sandwich-matrix  $\P$ is normalised, $(1,1,1)$ is the identity of $\Gamma$.

Suppose $x,y\in\Gamma$. Then we have the following identities:
\begin{eqnarray}
\label{eq:leech1}
(\lambda,g,i)x=(\lambda,g,1)x=(\lambda,g,1)x(1,1,1),\\
\label{eq:leech2}
x(\lambda,g,i)=x(1,g,i)=(1,1,1)x(1,g,i)\\
\label{eq:leech3}
((\lambda,g,1)x(1,h,j))\1=(\lambda,p_{j\lambda}\1 h\1,1)x\1(1,g\1 p_{j\lambda}\1,j)\\
\label{eq:leech4}
((\lambda,g,i)x)\1=(\lambda,1,1)x\1(1,g\1,1),\\
\label{eq:leech5}
(x(\lambda,g,i))\1=(1,g\1,1)x\1 (1,1,i),\\
\label{eq:leech6}
x(\lambda,g,i)y=x(1,g,1)y,\\
\label{eq:leech7}
(x(\lambda,g,i)y)\1 =y\1 (1,g\1,1)x\1.
\end{eqnarray}

The proof of the equalities~(\ref{eq:leech1},~\ref{eq:leech2}) is straightforward.

Let us prove~(\ref{eq:leech3}). Suppose $x=(1,g_x,1)$, then 
\begin{multline*}
((\lambda,g,1)(1,g_x,1)(1,h,j))\1=(\lambda,gg_xh,j)\1=(\lambda,p_{j\lambda}\1 (gg_xh)\1 p_{j\lambda}\1,j)\\
=(\lambda,p_{j\lambda}\1 h\1 g_x\1 g\1 p_{j\lambda}\1,j).
\end{multline*}
On the other hand,
\begin{multline*}
(\lambda,p_{j\lambda}\1 h\1,1)(1,g_x,1)\1(1,g\1 p_{j\lambda}\1,1)=(\lambda,p_{j\lambda}\1 h\1,1)(1,g_x\1,1)(1,g\1 p_{j\lambda}\1,1)\\
=(\lambda,p_{j\lambda}\1 h\1 g_x\1 g\1 p_{j\lambda}\1,j).
\end{multline*}

The equalities~(\ref{eq:leech4},~\ref{eq:leech5}) immediately follows from ~(\ref{eq:leech1},~\ref{eq:leech2},~\ref{eq:leech3}). Let us prove~(\ref{eq:leech6},~\ref{eq:leech7}). Suppose $x=(1,g_x,1)$, $y=(1,g_y,1)$, then
\[
x(\lambda,g,i)y=(1,g_x,1)(\lambda,g,i)(1,g_y,1)=(1,g_xgg_y,1),
\]
\[
(x(\lambda,g,i)y)\1=(1,g_xgg_y,1)\1=(1,(g_xgg_y)\1,1).
\]
Conversely,
\[
x(1,g,1)y=(1,g_x,1)(1,g,1)(1,g_y,1)=(1,g_xgg_y,1),
\]
\[
y\1(1,g\1,1)x\1 =(1,g_y\1,1)(1,g\1,1)(1,g_x\1,1)=(1,g_y\1 g\1 g_x\1,1)=(1,(g_xgg_y)\1,1),
\]
which proves~(\ref{eq:leech6},~\ref{eq:leech7}).

\begin{lemma}
\label{l:S-eq_dom->G-eq_dom}
If a c.s. semigroup $S=(G,\P,\Lambda,I)$ is an e.d. in the language $\LL_S$, then the group $G$ is an e.d. in the language $\LL_G$. 
\end{lemma}
\begin{proof}
Since $S$ is an e.d., the set $\M=\{(x,y)|x=(1,1,1)\mbox{ or }y=(1,1,1)\}\subseteq S^2$ is algebraic. In other words, there exists a system $\Ss$ with $\V_S(\Ss)=\M$. 

Below we define a system $\Ss^\pr$ which is equivalent to $\Ss$ over the group $\Gamma$, and all constants of $\Ss^\pr$ belong to $\Gamma$.

Applying~(\ref{eq:leech3},~\ref{eq:leech4},~\ref{eq:leech5},~\ref{eq:leech7}) to the system $\Ss$, one can obtain a system $\Ss_1$ such that:
\begin{enumerate}
\item $\Ss_1$ is equivalent to $\Ss$ over $\Gamma$;
\item the inversion ${}^{-1}$ in any equation from $\Ss_1$ is applied only to variables: $x\1, y\1$. 
\end{enumerate}  

By~(\ref{eq:leech6}) the system $\Ss_1$ is reduced to $\Ss_2$ such that any constant $\c$ in any $t(x,y)=s(x,y)\in\Ss_2$ belongs to $\Gamma$ if $\c$ is neither first nor last symbol in the terms $t(x,y),s(x,y)$. Obviously, $\Ss_2$ is equivalent to $\Ss$ over the group $\Gamma$. 

Consider an equation $(\lambda,g,i)t^\pr(x,y)=s(x,y)\in\Ss_2$, where the left part of this equation begins with a constant $\c=(\lambda,g,i)\notin\Gamma$ (similarly, one can consider an equation that has a right part ending by a constant $\c\notin\Gamma$). According to~(\ref{eq:leech1}), one can put $i=1$. Since the system $\Ss_2$ is consistent over $\Gamma$, the term $s(x,y)$ begins with either a constant $(\lambda,h,j)$ or a variable with $\lambda=1$.

Suppose $s(x,y)=(\lambda,h,j)s^\pr(x,y)$ (similarly, one can consider $s(x,y)$ which begins with a variable). It is directly checked that the equation 
\[
(\lambda,g,1)t^\pr(x,y)=(\lambda,h,j)s^\pr(x,y)
\]
is equivalent to
\[
(1,g,1)t^\pr(x,y)=(1,h,1)s^\pr(x,y).
\]
Thus, any equation of $\Ss_2$ is equivalent to a system $\Ss^\pr$ whose constants belong to the subgroup  $\Gamma$. Therefore, $\V_\Gamma(\Ss^\pr)=\{(x,y)|x=(1,1,1)\mbox{ or }y=(1,1,1)\}\subseteq \Gamma^2$, and, by Theorem~\ref{th:criterion_for_groups}, the group $\Gamma$ is an e.d. in the language $\LL_\Gamma$. Finally, the isomorphism between the groups $\Gamma,G$ proves the lemma.
\end{proof}

Below the element $(1,1,1)\in\Gamma$ is denoted by $\one$ for shortness.

\begin{lemma}
\label{l:reduce_x=y}
Suppose the sandwich-matrix $\P$ of a c.s. semigroup $S=(G,\P,\Lambda,I)$ is nonsingular. Then the equation $x=y$ is equivalent to a system of the form $\Ss_{=}(x,y)=\{t_i(x,y)=\one |i\in\mathcal{I}\}$.
\end{lemma}
\begin{proof}
Denote by $\M_{=}$ the solution set of $x=y$, i.e. $\M_{=}=\{(x,x)|x\in S\}$.

By Lemma~\ref{l:exists_dist_term}, for any $(s,s^\pr)\notin\M_{=}$ there exists a term $t_{(s,s^\pr)}(x)$ separating $s,s^\pr$.

Put
\[\Ss_{=}(x,y)=\{\one x(\one y\one )\1=\one \}\bigcup_{(s,s^\pr)\notin\M_{=}}\{t_{(s,s^\pr)}(x)(t_{(s,s^\pr)}(y))\1=\one \}.
\]

Let us show that $\M_=$ is the solution set of $\Ss_{=}(x,y)$, and $P=((\mu,g,j),(\mu,g,j))\in\M_=$. 

Since  
\[\one (\mu,g,j)(\one (\mu,g,j)\one )\1=(1,g,j)(1,g\1,1)=\one ,
\]
the point $P$ satisfies the first equation of $\Ss_=$.
For the other equations of $\Ss_=$ we have (apply the  formula~(\ref{eq:dist_term}) to the term $t_{(s,s^\pr)}(x)$)
\begin{multline*}
t_{(s,s^\pr)}((\mu,g,j))(t_{(s,s^\pr)}((\mu,g,j)))\1=(1,1,i)(\mu,g,j)(\lambda,1,1)((1,1,i)(\mu,g,j)(\lambda,1,1))\1=\\
(1,p_{i\mu}gp_{j\lambda},1)((1,p_{i\mu}gp_{j\lambda},1))\1=(1,p_{i\mu}gp_{j\lambda},1)(1,(p_{i\mu}gp_{j\lambda})\1,1)=\one .
\end{multline*}

Thus, any point of $\M_=$ belongs to the solution set of $\Ss_=$. Conversely, if $(s,s^\pr)\notin\M_=$, Let $(1,h,1)$, $(1,h^\pr,1)$ denote the values of $t_{(s,s^\pr)}(x)$ at $s,s^\pr$, respectively. We obtain
\[
t_{(s,s^\pr)}(s)(t_{(s,s^\pr)}(s^\pr))\1=
(1,h,1)((1,h^\pr,1))\1=(1,h(h^\pr)\1,1)\neq\one ,
\]
since $h\neq h^\pr$. Thus, the equation $t_{(s,s^\pr)}(x)(t_{(s,s^\pr)}(y))\1=\one \in\Ss_=$ is not satisfied by the point $(s,s^\pr)$.

Finally, we obtain $\V_S(\Ss_=)=\M_=$, therefore the system $\Ss_=(x,y)$ is equivalent to the equation $x=y$.
\end{proof}

\begin{lemma}
\label{l:x=1_y=1_implies_domain}
If the set $\M=\{(x,y)|x=(1,1,1)\mbox{ or } y=(1,1,1)\}$ is algebraic over a c.s. semigroup  $S=(G,\P,\Lambda,I)$ with a nonsingular $\P$, then $S$ is an e.d. in the language $\LL_S$.
\end{lemma}
\begin{proof}
By Theorem~\ref{th:about_M} it suffices to prove the existence of a system  $\Ss(x_1,x_2,x_3,x_4)$ with the solution set $\M_{sem}$.

By Lemma~\ref{l:reduce_x=y}, the equations $x_1=x_2$, $x_3=x_4$ are respectively equivalent to the systems $\Ss_=(x_1,x_2)=\{t_i(x_1,x_2)=\one |i\in\mathcal{I}\}$, $\Ss_=(x_3,x_4)=\{t_i(x_3,x_4)=\one |i\in\mathcal{I}\}$.

Applying the distributivity law for algebraic sets, we obtain 
\[
\M_{sem}=\bigcap_{i,j\in\mathcal{I}}\V_S(t_i(x_1,x_2)=\one) \cup \V_S( t_j(x_3,x_4)=\one ).
\] 

By the condition of the lemma, there exists a system $\Ss^\pr(x,y)$ with the solution $\M$. Hence, the set $\V_S(t_i(x_1,x_2)=\one \cup\V_S(t_j(x_3,x_4)=\one)$ coincides with the solution set of the system $\Ss^\pr(t_i(x_1,x_2),t_j(x_3,x_4))$, and
\[\M_{sem}=\bigcap_{i,j\in\mathcal{I}}\V_S(\Ss^\pr(t_i(x_1,x_2),t_j(x_3,x_4))).\]

Finally, the set $\M_{sem}$ equals to the solution set of the system 
\[\Ss(x_1,x_2,x_3,x_4)=\bigcup_{i,j\in\mathcal{I}}\Ss^\pr(t_i(x_1,x_2),t_j(x_3,x_4)).\]
Thus, $\M_{sem}$ is algebraic.
\end{proof}

\begin{lemma}
\label{l:x=1_y=1_is_algebraic}
Suppose the sandwich-matrix $\P$ of a c.s. semigroup $S=(G,\P,\Lambda,I)$ is nonsingular, the group $G$ is an e.d. in the group language $\LL_G$. Then the set $\M=\{(x,y)|x=(1,1,1)\mbox{ or } y=(1,1,1)\}$ is algebraic over $S$ in the language $\LL_S$.
\end{lemma}
\begin{proof}
Let $\one \neq s\in S$. By Lemma~\ref{l:exists_dist_term}, there exists a term of the form~(\ref{eq:dist_term}) separating $\one ,s$. This term is denoted below by $t_s(x)$. 

Put
\begin{multline*}
\Ss^\pr(x,y)=\bigcup_{g\in G}\{\one x(1,g\1,1)y(1,g,1)(\one x\one )\1(1,g\1,1)(\one y\one )\1(1,g,1)=\one \}\\
\bigcup_{(s,s^\pr)\notin\M,g\in G}\{t_s(x)(1,g\1,1)t_{s^\pr}(y)(1,g,1)(t_s(x))\1(1,g\1,1)(t_{s^\pr}(y))\1(1,g,1)=\one \}.
\end{multline*}

Let us show that $\V_S(\Ss^\pr(x,y))=\M$. 

Firstly, we prove $\M\subseteq\V_S(\Ss^\pr)$.

Let $P=((\mu,h,j),\one )\in\M$ (similarly, one can consider the point $(\one ,(\mu,h,j))\in\M$). We have
\begin{multline*}
\one (\mu,h,j)(1,g\1,1)\one (1,g,1)(\one (\mu,h,j)\one )\1(1,g\1,1)(\one \one \one )\1(1,g,1)=\\
\one (\mu,h,j)\one ((1,h\1,1))(1,g\1,1)(\one )(1,g,1)=\\
(1,h,1)(1,h\1,1)(1,g\1,1)(1,g,1)=\one .
\end{multline*}
In other words, the point $P$ satisfies the first group of equations of $\Ss^\pr$.

Since the terms $t_s(x),t_{s^\pr}(x)$ are of the form~(\ref{eq:dist_term}), we have $t_s(\one )=t_{s^\pr}(\one )=\one $. Let $t_s((\mu,h,j))=(1,f,1)$, $t_{s^\pr}((\mu,h,j))=(1,f^\pr,1)$.

Computing
\begin{multline*}
t_s((\mu,h,j))(1,g\1,1)t_{s^\pr}(\one )(1,g,1)(t_s((\mu,h,j)))\1(1,g\1,1)(t_{s^\pr}(\one ))\1(1,g,1)=\\
(1,f,1)(1,g\1,1)\one (1,g,1)(1,f\1,1)(1,g\1,1)\one (1,g,1)=\one
\end{multline*}
we see that the point $P$ satisfies the second group of equations of $\Ss^\pr$. Finally, we obtained $P\in\V_S(\Ss^\pr)$.

Second, we prove $\M\supseteq\V_S(\Ss^\pr)$.

Let $Q=(s,s^\pr)\notin\M$. Assume that $Q$ satisfies all equations 
\[
t_s(x)(1,g\1,1)t_{s^\pr}(y)(1,g,1)(t_s(x))\1(1,g\1,1)(t_{s^\pr}(y))\1(1,g,1)=\one 
\]
for all $g\in G$.

By the choice of the terms $t_s(x),t_{s^\pr}(x)$ there exist elements $1\neq f,f^\pr\in G$ with $t_s(s)=(1,f,1)$, $t_{s^\pr}(s^\pr)=(1,f^\pr,1)$.

Therefore, for all $g\in G$ we have
\[
t_s(s)(1,g\1,1)t_{s^\pr}(s^\pr)(1,g,1)(t_s(s)\1(1,g\1,1)(t_{s^\pr}(s^\pr))\1(1,g,1)=\one ,
\]
\[
(1,f,1)(1,g\1,1)(1,f^\pr,1)(1,g,1)(1,f\1,1)(1,g\1,1)(1,(f^\pr)\1,1)(1,g,1)=\one,
\]
\[
(1,fg\1 f^\pr gf\1 g\1(f^\pr)\1 g,1)=\one ,
\]
\[fg\1 f^\pr gf\1 g\1(f^\pr)\1 g=1\Leftrightarrow[f,(f^\pr)^g]=1,
\]
where $[f,(f^\pr)^g]$ is the commutator of $f$ and $f^\pr$ conjugated by $g$ in the group $G$.

Since the last equality holds for all $g\in G$, the elements $f,f^\pr$ are zero-divisors in the group $G$. Thus, by Theorem~\ref{th:zero_divisors}, the group $G$ is not an e.d. that contradicts with the condition of the lemma. 

\end{proof}

\begin{theorem}
\label{th:criterion}
A c.s. semigroup $S=(G,\P,\Lambda,I)$ is an e.d. in the language $\LL_S=\{\cdot,{ }\1\}\cup\{s|s\in S\}$ iff the following hold:
\begin{enumerate}
\item the sandwich-matrix $\P$ is non-singular; 
\item the structural group $G$ is an e.d. in the group language $\LL_G$. 
\end{enumerate}
\end{theorem}
\begin{proof}
The ``only if '' part of the statement immediately follows from Lemmas~\ref{l:singular->not_ED},~\ref{l:S-eq_dom->G-eq_dom}. Let us prove the converse.

First, by Lemma~\ref{l:x=1_y=1_is_algebraic}, we obtain that the set $\M=\{(x,y)|x=(1,1,1)\mbox{ or } y=(1,1,1)\}$ is algebraic over $S$ over the language $\LL_S$. Second, one can apply Lemma~\ref{l:x=1_y=1_implies_domain}, which proves the theorem.
\end{proof}

Theorems~\ref{th:free_CSS_structure},~\ref{th:free_product_CSS_structure}, Corollaries~\ref{cor:free_group_is_ED},~\ref{cor:free_product_of_groups_is_ED} and the criterion provided in Theorem~\ref{th:criterion} allow us to prove easily the following statements.

\begin{corollary}
Any free c.s. semigroup $S=F_{css}(X)$ of the rank $n\geq 2$ is an e.d. in the language $\LL_S$.
\end{corollary}

\begin{corollary}
The free product $S=S_1\ast S_2$ of arbitrary c.s. semigroups $S_1,S_2$ is an e.d. in the language $\LL_S$.   
\end{corollary}

The next statement follows from Lemma~\ref{l:reduce_x=y}.

\begin{corollary}
Any algebraic set $Y\subseteq S^n$ over a c.s. semigroup $S=(G,\P,\lambda,I)$ with a non-singular sandwich-matrix $\P$ is defined by a system $\Ss=\{t_i(X)=\one |i \in \mathcal{I}\}$.
\end{corollary}
\begin{proof}
Suppose an algebraic set $Y$ is the solution set of a system $\Ss^\pr=\{t_i(X)=s_i(X)|i \in \mathcal{I}\}$. 

According to Lemma~\ref{l:reduce_x=y}, the equation $x=y$ is equivalent to the system $\Ss_=(x,y)$ of the form $\{r_j(x,y)=\one |j\in\mathcal{J}\}$. Therefore every equation $t_i(X)=s_i(X)\in\Ss^\pr$ is equivalent to the system $\Ss_=(t_i(X),s_i(X))$. Thus, $\Ss^\pr$ is equivalent to
\[
\Ss=\bigcup_{i\in\mathcal{I}}\Ss_=(t_i(X),s_i(X))=
\bigcup_{i\in\mathcal{I},\\ j\in\mathcal{J}}r_j(t_i(X),s_i(X)).
\]
\end{proof}

\section{Generalisations to other semigroup languages}
\label{sec:generalization}
A c.s. semigroup $S$ can be considered in any language $\LL_T=\{\cdot,{}\1\}\cup\{s|s\in T\}$, where $T$ is a subsemigroup of $S$. An $\LL_T$-term is an $\LL_S$-term whose constants belong to the subsemigroup $T$. 

Similarly, any group $G$ can be considered in a language $\LL_H=\{\cdot,{}\1\}\cup\{g|g\in H\}$, where $H$ is a subgroup of $G$. 

The notions of an equation, algebraic set, and equational domain in the language $\LL_T$ ($\LL_H$) are similar to the corresponding definitions in the language $\LL_S$ (respectively, $\LL_G$). 

Since $\LL_T\subseteq\LL_S$ ($\LL_H\subseteq\LL_G$), any algebraic set in the language $\LL_T$ (respectively, $\LL_H$) is algebraic in $\LL_S$ ($\LL_G$). Therefore, any equational domain in the language $\LL_T$ ($\LL_H$) is an e.d. in $\LL_S$ (respectively, $\LL_G$). It follows that equational domains in the language $\LL_T$ should satisfy stronger conditions than those in Theorem~\ref{th:criterion}. 

Similarly, one can show that the analogue of Theorem~\ref{th:zero_divisors} for equational domains in the language $\LL_H$ should contain stronger conditions. Actually, in~\cite{uniTh_IV} the definition of a zero-divisor with respect to a subgroup $H\leq G$ was given: an element $x\neq 1$ of a group $G$ is an \textit{$H$-zero-divisor} if there exists $1\neq y\in G$ such that $[x,y^h]=1$ for all $h\in H$.

The next theorem describes equational domains in the group language $\LL_H$.

\begin{theorem}\textup{\cite{uniTh_IV}}
\label{th:zero_divisors_for_H}
A group $G$ is an equational domain in the language $\LL_H$ iff it does not contain $H$-zero-divisors.
\end{theorem}

Since the class of c.s. semigroups is a variety, a subsemigroup $T$ of a c.s. semigroup  $S=(G,\P,\Lambda,I)$ has the representation $T=(H,\P^\pr,\Lambda^\pr,I^\pr)$, where $H\leq G$, $\Lambda^\pr\subseteq\Lambda$, $I^\pr\subseteq I$ and $\P^\pr$ is a submatrix of $\P$, $\P^\pr=(p_{i\lambda}|i\in I^\pr,\lambda\in\Lambda^\pr)$. 

For the language $\LL_T$ we obtain the next criterion. 

\begin{theorem}
A c.s. semigroup $S=(G,\P,\Lambda,I)$ is an e.d. in the language $\LL_T=\{\cdot,{ }\1\}\cup\{s|s\in T\}$ iff the following conditions hold:
\begin{enumerate}
\item the sandwich-matrix $\P$ is non-singular; 
\item the structural group $G$ is an e.d. in the group language $\LL_{H}$, where $H$ is the structural group of $T$. 
\end{enumerate}
\end{theorem}
\begin{proof}
The ``only if'' part of the statement follows from Theorem~\ref{th:criterion}, since any e.d. in the language $\LL_T$ is an e.d. in $\LL_S$.

To prove the converse statement one can apply Lemmas~\ref{l:reduce_x=y},~\ref{l:x=1_y=1_implies_domain},~\ref{l:x=1_y=1_is_algebraic} with the subgroup $H$ instead of $G$.
\end{proof}

The author was supported by RFFI grant 14-01-00068

The information of the author:

Artem N. Shevlyakov

Sobolev Institute of Mathematics

644099 Russia, Omsk, Pevtsova st. 13

Phone: +7-3812-23-25-51.

e-mail: \texttt{a\_shevl@mail.ru}
\end{document}